\documentclass{amsart}

\usepackage{amssymb,amsmath,amscd,amsthm,enumerate}

\date{\today}

\newcommand{\Z}{{\mathbb Z}}
\newcommand{\R}{{\mathbb R}}
\newcommand{\C}{{\mathbb C}}

\newcommand{\D}{{\mathbb D}}

\newcommand{\E}{{\mathcal E}}

\newcommand{\B}{{\mathrm B}}
\newcommand{\St}{{\mathrm S}}

\newcommand{\loc}{{\mathrm{loc}}}

\newcommand{\PT}{{\mathrm{PT}}}
\newcommand{\LP}{{\mathrm{LP}}}

\newtheorem{theorem}{Theorem}[section]
\newtheorem{lemma}[theorem]{Lemma}

\newtheorem{coro}[theorem]{Corollary}
\newtheorem*{conj}{Conjecture}
\newtheorem*{quest}{Question}

\theoremstyle{definition}

\theoremstyle{definition}
\newtheorem*{defi}{Definition}

\renewcommand{\Im}{\mathrm{Im} \, }

\newcommand{\tr}{\mathrm{tr} }

\newcommand{\per}{{\mathrm{per}}}

\allowdisplaybreaks

\begin{document}

\title{Spectral Homogeneity of Limit-Periodic Schr\"odinger Operators}

\author[J.\ Fillman]{Jake Fillman}

\address{Department of Mathematics, Rice University, 6100 Main Street, MS-136, Houston, TX 77005, USA}

\email{jdf3@rice.edu}

\thanks{J.\ F.\ was supported in part by NSF grants DMS--1067988 and DMS--1361625.}

\author[M.\ Lukic]{Milivoje Lukic}

\address{Department of Mathematics, University of Toronto, Bahen Centre, 40 St.\ George St., Toronto, Ontario, CANADA, M5S 2E4 and  Department of Mathematics, Rice University, 6100 Main Street, MS-136, Houston, TX 77005, USA}

\email{mlukic@math.toronto.edu}

\thanks{M.\ L.\ was supported in part by NSF grant DMS--1301582. M.\ L.\ would also like to thank the Isaac Newton Institute for Mathematical Sciences, Cambridge, for support and hospitality during the programme ``Periodic and Ergodic Spectral Problems" where part of this work was undertaken.}

\maketitle

\begin{abstract}
We prove that the spectrum of a limit-periodic Schr\"odinger operator is homogeneous in the sense of Carleson whenever the potential obeys the Pastur--Tkachenko condition. This implies that a dense set of limit-periodic Schr\"odinger operators have purely absolutely continuous spectrum supported on a homogeneous Cantor set. When combined with work of Gesztesy--Yuditskii, this also implies that the spectrum of a Pastur--Tkachenko potential has infinite gap length whenever the potential fails to be uniformly almost periodic.
\end{abstract}

\section{Introduction}

We study continuum Schr\"odinger operators, that is, operators on $L^2(\R)$ of the form $H = H_V$, where
\begin{equation} \label{so:def}
H_V\phi
=
-\phi'' + V\phi
\end{equation}
and $V$ is a real-valued funtion on $\R$. It is well known that \eqref{so:def} defines a self-adjoint operator on a dense linear subspace of $L^2(\R)$ under appropriate restrictions on the potential, $V$. We are particularly interested families of \emph{ergodic dynamically defined potentials}. More specifically, an ergodic potential is one which may be written in the form
$$
V_\omega(x)
=
f(\tau_x \omega),
\quad
x \in \R, \, \omega \in \Omega,
$$
where $(\Omega,\mathcal B, \mu)$ is a probability measure space, $\tau$ is a $\mu$-ergodic flow on $\Omega$, and $f:\Omega \to \R$ is a suitable measurable function. In this setting, one may use tools from dynamics to probe the operators as a family, proving statements which hold for $\mu$-almost every $\omega \in \Omega$.

 For operators of the form \eqref{so:def}, there are essentially two broad classes of spectral-theoretic results: \emph{direct spectral theory}, which deduces characteristics of the spectrum and spectral measures of $H_V$ for a given $V$ or a given class of $V$'s, and \emph{inverse spectral theory}, which fixes spectral information and attempts to describe the space of potentials which exhibit such spectral data. In rare instances, one is able to obtain a result which goes in both directions simultaneously -- such results are called gems of spectral theory in \cite{simszego}. In this paper, we will consider the class of Pastur--Tkachenko potentials, which have been heavily studied from the inverse point of view \cite{PT1, PT2}; see also~\cite{Chul81,MC84}. To begin, let us recall the definitions of limit-periodic and almost-periodic potentials, and the Pastur--Tkachenko class:

\begin{defi}
A potential $V:\R \to \R$ is (\emph{Stepanov}) \emph{limit-periodic} if there exist periodic $V_n \in L^2_{\loc}(\R)$ such that
$$
\lim_{n \to \infty} \| V - V_n \|_{\St}
=
0,
$$
where $\lVert \cdot \rVert_{\St}$ stands for the Stepanov norm,
\[
\| f \|_{\St}
=
\sup_{x\in\mathbb{R}} \left(
\int_{x}^{x+1} \!|f(t)|^2 \, dt
\right)^{1/2}.
\]
We will denote the space of all Stepanov limit-periodic functions on $\R$ by $\LP(\R)$. On that space, we will also use the \emph{Besicovitch norm}, defined by
$$
\| f \|_{\B}
=
\lim_{T \to \infty} \left(
\frac{1}{2T}\int_{-T}^T \!|f(t)|^2 \, dt
\right)^{1/2}.
$$
It is easy to see that one always has $\|f\|_{\B} \le \| f \|_{\St}$. Notice further that any Stepanov limit-periodic potential is \emph{Stepanov} \emph{almost-periodic}, i.e., it is a limit of trigonometric polynomials with respect to $\| \cdot \|_{\St}$. If $V$ is almost-periodic, then it can be realized in the ergodic context with
$$
\Omega
=
\mathrm{hull}(V)
=
\overline{\{ \tau_x V : x \in \R \}},
\quad
\tau_x \omega
=
\omega(\cdot + x),
\quad
f(\omega)
=
\omega(0),
$$
where the closure is meant to be taken with respect to $\| \cdot \|_{\St}$.

\bigskip

We say that a limit-periodic potential $V$ satisfies the \emph{Pastur--Tkachenko condition} if its periodic approximants can be chosen such that
\begin{equation} \label{eq:pt:def}
\lim_{n \to \infty} e^{b T_{n+1}} \| V - V_n \|_{\St}
=
0
\text{ for every } b > 0,
\end{equation}
where $T_n$ denotes the period of $V_n$, and $T_n$ divides $T_{n+1}$ for all $n$; compare~\cite{Chul81,MC84,PT1,PT2}. We denote by $\PT(\R)$ the set of all potentials $V \in \LP(\R)$ which obey the Pastur--Tkachenko condition. Notice that $V \in \PT(\R)$ means that $V$ also satisfies \eqref{eq:pt:def} with $\| \cdot \|_{\St}$ replaced by $\| \cdot \|_{\B}$. 
\end{defi}

Our main goal in this paper is to prove that the spectrum of every potential in the class $\PT(\R)$ is homogeneous in the sense of Carleson. Loosely speaking, this says that the spectrum has a uniform positive density with respect to Lebesgue measure around each of its points. The precise formulation follows.

\begin{defi}
We say that a set $\Sigma \subseteq \R$ is $\tau$-\emph{homogeneous} if there exists $\delta_0 > 0$ such that
$$
|B_\delta(x) \cap \Sigma|
\geq \tau \delta
\text{ for all } 
x \in \Sigma 
\text{ and } 
0<\delta\leq\delta_0,
$$
where $B_\delta(x) = (x - \delta, x + \delta)$ denotes the $\delta$-neighborhood of $x$ and $|\cdot|$ denotes Lebesgue measure (compare \cite{carleson83}).
\end{defi}

\begin{theorem} \label{t:pt}
If $V \in \PT(\R)$, then $\sigma(H_V)$ is $\tau$-homogeneous for every $\tau \in (0,1)$.
\end{theorem}

As a consequence of Theorem~\ref{t:pt}, we easily recover a continuum Schr\"odinger analog of a strengthened version of the main result of \cite{F14}.

\begin{coro} \label{coro:homogdense}
For each $0 < \tau < 1$, denote by $\mathcal H_\tau$ the set of $V \in \LP(\R)$ such that $H_V$ has purely absolutely continuous spectrum and $\sigma(H_V)$ is a $\tau$-homogeneous Cantor set. Then $\mathcal H_{1^-} = \bigcap_{0<\tau<1}\mathcal H_\tau$ is dense in $\LP(\R)$ \rm{(}with respect to the topology induced by $\| \cdot \|_{\St}${\rm)}.
\end{coro}

\begin{proof}
It is easy to see that $\PT(\R)$ is dense in $\LP(\R)$. In \cite[Theorem 17.2]{PasturFigotin}, the spectrum is encoded in terms of Marchenko--Ostrovski heights $\lvert \kappa_r\rvert$ and it is easy to see that for a dense set of $V \in \PT(\R)$, all $\lvert \kappa_r \rvert$ are strictly positive, and therefore all gaps are open and the spectrum is a Cantor set. By Theorem~\ref{t:pt} and \cite{PT1,PT2}, $\PT(\R) \subseteq \mathcal H_{1^-}$. Thus, $\mathcal H_{1^-}$ is dense in $\LP(\R)$, as desired.
\end{proof}

\noindent \textit{Remark.} Theorem~\ref{t:pt} and Corollary~\ref{coro:homogdense} improve the results of \cite{F14} in two senses. First, for each fixed $\tau < 1$, \cite{F14} constructs a dense set of examples with $\tau$-homogeneous spectrum, while the previous result constructs a dense set which works for all $\tau<1$ simultaneously. Secondly, \cite{F14} does not construct explicit examples of aperiodic limit-periodic potentials with homogeneous spectrum, while the Pastur--Tkachenko class is very explicit.
\bigskip 

The motivation for studying spectral homogeneity arises from inverse spectral problems in the regime of absolutely continuous (a.c.)\ spectrum. Specifically, it is well-known that the presence of nontrivial a.c.\ spectrum in one dimension places rather strict conditions on the potential; compare \cite{kotani89, remling2011}. This led to the popularity of the following conjecture:

\begin{conj}[Kotani--Last]
If an ergodic Schr\"odinger operator has nonempty absolutely continuous spectrum, then the potential is  almost-periodic.
\end{conj}

The conjecture is false -- Avila constructed counterexamples \cite{avilaJAMS}, while Damanik, Volberg and Yuditskii constructed counterexamples using different techniques \cite{DamYud, VolYud}. 

On the other hand, if $\Sigma$ is a homogeneous closed subset of $\R$ with finite gap length, then the space of potentials which have spectrum $\Sigma$ and are reflectionless thereupon is known to consist of almost-periodic operators by a theorem of Sodin and Yuditskii \cite{SY95}; moreover, Gesztesy and Yuditskii have proved that the spectral measures of reflectionless Schr\"odinger operators with homogeneous spectrum are purely absolutely continuous \cite{GY06}. There are analogous results for the inverse spectral theory of Jacobi and CMV matrices in \cite{poltrem09, SY97} and \cite{GZ09}, respectively. One should note that results of this form do not always hold if the spectrum is not homogeneous; compare \cite{DamYud, VolYud, Yud2012}. The inverse spectral results merely require that $\overline{\C} \setminus \Sigma$ is a Widom domain with a direct Cauchy theorem. Homogeneity of $\Sigma$ is a pleasantly explicit sufficient condition which guarantees that this happens; see \cite{Yud2012}.

In light of this, one wonders what kinds of restrictions the presence of absolutely continuous spectrum places on the spectrum itself.

\begin{quest}
If $V$ is almost-periodic, and $H_V$ has nonempty absolutely continuous spectrum, must $\sigma(H_V)$ be homogeneous?
\end{quest}

Although this is a question about inverse spectral theory, one can attempt to study it from a direct spectral perspective by seeing whether or not examples known to have a.c.\ spectrum also have homogeneous spectrum, which is what we do here in the limit-periodic class. For results in this vein in the setting of analytic quasi-periodic operators, see \cite{DGL,GDSV}.

\bigskip

Questions about absolutely continuous spectrum in the one-dimensional ergodic case are closely tied to reflectionlessness, whose definition we briefly recall.

\begin{defi}
Suppose $V \in L^1_{\loc}(\R)$ is such that $H_V$ is in the limit-point case at $+\infty$ and $-\infty$. Then, for any $z \in \C \setminus \R$, there exist $\psi_\pm = \psi_\pm(x,z)$ which solve the initial value problem
$$
-\psi'' + V \psi = z\psi,
\quad
\psi(0,z) = 1
$$
and such that $\psi_\pm(\cdot,z)$ is $L^2$ near $\pm \infty$. The Weyl $m$-functions are given by
$$
m^\pm(x,z)
=
\mp
\left( \frac{\psi_\pm'(x,z)}{\psi_\pm(x,z)} \right).
$$
One says that $V$ is \emph{reflectionless} on $\Sigma \subseteq \R$ if
$$
m^+(x, E + i0)
=
-\overline{m^-(x,E+i0)}
$$
for all $x \in \R$ and Lebesgue almost-every $E \in \Sigma$, and one simply calls $V$ \emph{reflectionless} if one may take $\Sigma = \sigma(H_V)$.
\end{defi}

While the work of Pastur--Tkachenko does not explicitly conclude that potentials in $\PT(\R)$ are reflectionless, this follows readily from their work. For instance, denoting Weyl solutions of $H_V$ at $z\notin \sigma(H_V)$ by $\psi_\pm(\cdot, z)$, with the normalization $\psi_\pm(0,z)=1$, it follows from the arguments in \cite[Section~17]{PasturFigotin} that for almost every $E \in \sigma(H_V)$, the boundary values $\psi_\pm(x,E+i0)$ exist and obey
\begin{equation}\label{eq:pmsymmetry}
\psi_-(x,E+i0) = \overline{\psi_+(x,E+i0)}
\end{equation}
(since these limits are also limits of Weyl solutions for the periodic approximants, and those obey the condition \eqref{eq:pmsymmetry}), and property \eqref{eq:pmsymmetry} implies that $H_V$ is reflectionless. For different characterizations of reflectionlessness, see Breuer--Ryckman--Simon~\cite{BRS09}, Jak\v si\' c--Landon--Panati~\cite{JLP14}.

Potentials which satisfy a Pastur--Tkachenko condition are Stepanov almost-periodic, but not necessarily uniformly almost periodic (i.e.\ a uniform limit of trigonometric polynomials). By combining our theorem with work of Gesztesy-Yuditskii \cite{GY06}, we harvest the following interesting corollary:

\begin{coro}
Given $V \in \PT(\R)$, denote
$$
\Sigma
=
\sigma(H_V)
=
[E_0,\infty) \setminus \bigcup_j (a_j, b_j).
$$
If $V$ is not uniformly almost periodic {\rm(}and in particular, if $V$ is not continuous{\rm)}, then the spectrum has infinite gap length, i.e.
$$
\sum_j (b_j - a_j)
=
\infty.
$$
\end{coro}

\begin{proof}
By way of contraposition, assume $\Sigma$ has finite gap length. Since $V$ is reflectionless by the foregoing remarks and $\Sigma$ is homogeneous by Theorem~\ref{t:pt}, then \cite[Theorem~2.5]{GY06} implies that $V$ is uniformly almost periodic.
\end{proof}

The organization of the paper is as follows. In Section~\ref{sec:floq}, we recall some relevant aspects of Floquet Theory for Hill's equation. In particular, we prove an estimate on the discriminant which gives us a useful lower bound on the length of a band of the spectrum of a periodic potential. Next, in Section~\ref{sec:bandstab}, we discuss the stability of periodic spectra under $L^2$ perturbations. These stability estimates replace the $L^\infty$ estimates that are exploited in \cite{F14}, and they comprise the key ingredient in proving effective step-by-step estimates on the measure of the spectra of periodic approximants. In Section~\ref{sec:specconv}, we describe the sense in which the periodic spectra converge to the spectrum associated with $V \in \PT(\R)$, which enables us to push the step-by-step estimates through to the limit.  Section~\ref{sec:proofs} combines these ingredients and proves Theorem~\ref{t:pt}. Finally, we discuss the extension of Theorem~\ref{t:pt} to Jacobi and CMV matrices in Section~\ref{sec:discrete}.

\section*{Acknowledgements}

J.\ F.\ thanks Quentin Funk for helpful conversations.

\section{Floquet-Bloch Theory for Periodic Potentials}\label{sec:floq}

Suppose $V \in L^2_{\loc}(\R)$ is $T$-periodic. Given $z \in \C$, let $y_D$ and $y_N$ denote the Dirichlet and Neumann solutions of the Schr\"odinger equation
\begin{equation} \label{eq:so}
-y'' + Vy = zy.
\end{equation}
Specifically, $y_D$ and $y_N$ solve \eqref{eq:so} subject to the initial conditions
$$
y_D(0) = y_N'(0) = 0,
\quad
y_D'(0) = y_N(0) = 1.
$$
The \emph{monodromy matrix} is defined by
\[
M_z = \begin{pmatrix}
y_N(T) & y_D(T) \\
y_N'(T)  & y_D'(T)
\end{pmatrix}
\]
and the \emph{discriminant} is given by $\Delta(z) = \tr(M_z)$. One has
$$
\sigma(H_V)
=
\{ E \in \R : |\Delta(E)| \leq 2 \}
=
\bigcup_{j=1}^\infty [\alpha_j,\beta_j],
$$
where $\alpha_1 < \beta_1 \leq \alpha_2 < \cdots$ denote the solutions of $\Delta = \pm 2$. We call the intervals $[\alpha_j, \beta_j]$ the \emph{bands} of the spectrum, while the intervals $(\beta_j, \alpha_{j+1})$ are called \emph{gaps}. Whenever $\beta_j = \alpha_{j+1}$, we say that the $j$th gap is \emph{closed}. Otherwise, $\beta_j < \alpha_{j+1}$ and the $j$th gap is said to be \emph{open}. The discriminant is strictly monotone on each band.

We prove an estimate on the discriminant, and the corresponding estimate on band lengths. These estimates are likely well-known, but we could not find a precise reference.

\begin{lemma} \label{l:discderivest}
 Let $V \in L^2_\loc(\R)$ be a $T$-periodic potential with corresponding discriminant $\Delta$, and let $E_0 = \min (0, \inf \sigma(H_V))$. Then
\begin{equation}\label{derivative}
\lvert \Delta'(E) \rvert 
\le 
\frac{C T^3}{T + \lvert E \rvert^{1/2}}  \exp\left( CT \left( \lVert V \rVert_{\B}^{1/2} +  \lvert E_0\rvert^{1/2} \right) \right)
\end{equation}
for all $E\in \sigma(H_V)$, for some universal constant $C > 0$ which does not depend on $E$, $T$, or $V$.
\end{lemma}

\begin{proof}
In this proof, $C$ will stand for different universal constants. Its value will increase only finitely many times over the course of the argument.

We begin by proving this for $T=\pi$. Denote $Q =\lVert V \rVert_{\B}$. It follows from \cite[Corollary 17.8]{PasturFigotin} that
\begin{equation}\label{PFeqn}
\lvert \Delta(\mu^2) \rvert \le C \exp\left( C Q^{1/2} + \pi \lvert \Im \mu \rvert \right).
\end{equation}
From the Cauchy differentiation formula on a circle of radius $1$,  for $f(\mu) = \Delta(\mu^2)$,
\[
\lvert f'(\mu) \rvert \le \int_0^{2\pi} \lvert f(\mu+e^{i\theta}) \rvert \frac{d\theta}{2\pi}
\]
so we conclude that
\begin{equation}\label{PFeqn1}
\lvert 2\mu \Delta'(\mu^2) \rvert \le C \exp\left( C Q^{1/2} + \pi \lvert \Im \mu \rvert + \pi \right).
\end{equation}
We can also conclude from \eqref{PFeqn} that
\[
\lvert \Delta(z) \rvert \le C \exp\left( C Q^{1/2} + 2\pi \right), \qquad \lvert z\rvert \le 2,
\]
and then by the Cauchy differentiation formula for $\Delta(z)$, similarly as above, that
\begin{equation}\label{PFeqn2}
\lvert \Delta'(z) \rvert \le C \exp\left( C Q^{1/2} + 2\pi \right), \qquad \lvert z\rvert \le 1.
\end{equation}
Using \eqref{PFeqn1} for $\lvert z \rvert >1$ and \eqref{PFeqn2} for $\lvert z \rvert \le 1$, we conclude that for all $z\in \mathbb{C}$,
\[
\lvert \Delta'(z) \rvert \le \frac{C}{1 + \lvert z \rvert^{1/2}} \exp\left( C Q^{1/2} + \pi \lvert \Im \sqrt z\rvert \right).
\]
If $E \in \sigma(H_V)$, then $\lvert\Im \sqrt E \rvert \le  \lvert E_0\rvert^{1/2}$, so we conclude that the statement of the lemma holds for $T = \pi$.

For a potential $V$ of arbitrary period $T$, introduce the rescaled $\pi$-periodic potential $V_\pi(x) = (T/\pi)^2 V(Tx/\pi)$. Its discriminant $\Delta_\pi(E)$ obeys $\Delta_\pi((T/\pi)^2 E) = \Delta(E)$, so the conclusion of the lemma for $V$ follows by applying the previous argument to $V_\pi$ and noting that $\lVert V_\pi \rVert_{\B} = (T/\pi)^2 \lVert V \rVert_{\B}$ and that $\nu_0$ scales in the same way.
\end{proof}

\begin{coro} \label{coro:breakct}
Let $C$ denote the universal constant from Lemma~\ref{l:discderivest}, and suppose $V \in L^2_{\loc}(\R)$ is $T$-periodic with $\|V\|_{\B} \leq Q$. Then the length of any band $[\alpha_j,\beta_j]$ of $\sigma(H_V)$ is bounded from below by 
$$
|\beta_j - \alpha_j|
\geq
4 C^{-1} e^{-CT(Q^{1/2}+\lvert E_0\rvert^{1/2})}(T+\lambda_0^{1/2}) T^{-3},
$$
with $E_0$ defined as in Lemma~\ref{l:discderivest} and
$$
\lambda_0
=
\min_{E \in [\alpha_j,\beta_j]} |E|.
$$

\end{coro}

\begin{proof}
This is a consequence of Lemma~\ref{l:discderivest} and the Mean Value Theorem, since $\lvert \Delta(\beta_j) - \Delta(\alpha_j) \rvert = 4$.
\end{proof}

\section{Stability of band edges} \label{sec:bandstab}

An important component of our proof of Theorem~\ref{t:pt} will be control of the stability of bands of periodic Schr\"odinger operators under an $L^2$ perturbation. Band edges of a $T$-periodic Schr\"odinger operator are precisely eigenvalues of the Schr\"odinger operator with periodic boundary conditions on $[0,2T]$. In this section, we prove a stability statement for those eigenvalues under an $L^2$ perturbation (note that, for $L^\infty$ perturbations, this is trivial). Pastur--Figotin \cite{PasturFigotin} contains an estimate of this form for the bottom of the spectrum, but we require an estimate for all eigenvalues. While estimates of this form are classical, we are not aware of this estimate appearing in the literature, so we provide a proof. Note that, for our application, it will be crucial that the estimate below features only the square root, and not a higher power, of the energy.

For $V \in L^2[0,T]$, denote by $H^\per_{V}$ the Schr\"odinger operator $-\Delta +V$ on $L^2[0,T]$ with periodic boundary conditions. Denote by $E_{n,V}$ the $n$-th eigenvalue (including multiplicity) of $H^\per_{V}$; in particular, $E_{1,V}$ is the ground state energy of $H^\per_{V}$. Finally, let $R_{\lambda,V}$ denote the resolvent of $H^\per_V$ at $\lambda \in \mathbb{C} \setminus \sigma(H^\per_V)$. In particular, $R_{\lambda,0}$ denotes the resolvent of the free Schr\"odinger operator $H_0^\per = -\Delta$ on $L^2[0,T]$ with periodic boundary conditions.

\begin{theorem} \label{Tbandstability}
Fix $Q > 0$. If $V_1, V_2 \in L^2[0,T]$ and $\lVert V_j \rVert_{\B} \le Q$ for $j=1$, $2$, then
\[
E_{1,V_j} \ge - C_1(Q + T^2 Q^2)
\]
and
\begin{equation} \label{eq:edgestability}
\lvert E_{n,V_1} - E_{n,V_2} \rvert \le C_1 (1+T^2 Q) (1+ T \lvert E_{n,V_2} \rvert^{1/2}) \lVert V_1 - V_2 \rVert_{\B}
\end{equation}
for some constant $C_1$ independent of $T$, $Q$, $V_1$, $V_2$.
\end{theorem}

The proof of the theorem will start with a comparison with the free operator, so our first lemma concerns the integral kernel of the resolvent of $H^\per_0$.

\begin{lemma} For $\lambda \notin \sigma(H^\per_0)$,
\begin{equation} \label{eq:pergreenid}
(R_{\lambda,0} f)(x)= \int_0^T \! G(\lambda;x,t) f(t) \, dt,
\end{equation}
where
\begin{equation} \label{eq:pergreendef}
G(\lambda;x,t) 
= 
- \frac 1{2\sqrt \lambda} \frac{\cos\left(\sqrt \lambda (\frac T2 - \lvert x-t \rvert)\right)}{\sin \left(\sqrt \lambda \frac T2\right)}.
\end{equation}
In particular, there is a constant $C \in (0,\infty)$ {\rm(}depending only on $T${\rm)} such that for all $\lambda < 0$ and all $x, t \in [0,T]$,
\[
\lvert G(\lambda;x,t) \rvert \le CT^{-1/2} (\lvert \lambda \rvert^{-1/2} + \lvert \lambda \rvert^{-1}),
\]
so $R_{\lambda,0}$ is a bounded operator from $L^2[0,T]$ to $L^\infty[0,T]$,
\begin{equation}\label{freeresolvent}
\lVert R_{\lambda,0} \rVert_{2,\infty} 
\le 
C (\lvert \lambda \rvert^{-1/2} + \lvert \lambda \rvert^{-1}).
\end{equation}
\end{lemma}

\begin{proof}
For $\lambda \neq 0$, let
\[
s_\lambda(x) =\lambda^{-1/2} \sin(\lambda^{1/2} x), \qquad c_\lambda(x) = \cos(\lambda^{1/2} x)
\]
The general solution of $- y'' - \lambda y = f$ is
\begin{equation} \label{eq:persol}
y(x) = A s_\lambda(x) + B c_\lambda(x) - \int_0^x \! f(t) s_\lambda(x-t) \, dt.
\end{equation}
Applying periodic boundary conditions, solving the resulting system for $A, B$ and inserting it back into $y(x)$ and simplifying, we obtain \eqref{eq:pergreenid} and \eqref{eq:pergreendef}. The estimates are then elementary.
\end{proof}

\begin{proof}[Proof of Theorem~\ref{Tbandstability}]
We will prove the estimates for $T=1$. After that, it is an immediate calculation to rescale the estimates to an arbitrary period $T$, as in the proof of Lemma~\ref{l:discderivest}.

Let $A = 16C$, where $C$ is the constant from \eqref{freeresolvent}, and suppose
\[
\lambda  
\leq 
- (A Q + A^2 Q^2).
\]
From \eqref{freeresolvent} we see that $V R_{\lambda,0}$ is a bounded operator on $L^2[0,T]$,
\[
\lVert V R_{\lambda,0} \rVert_{2,2} \le \lVert V \rVert_{\B}  \lVert R_{\lambda,0}  \rVert_{2,\infty} 
\le 
Q C (\lvert \lambda\rvert^{-1/2} + \lvert\lambda\rvert^{-1}) \le \frac {2C}A 
= 
\frac 18,
\]
so the Neumann series
\[
 \sum_{n=0}^\infty R_{\lambda,0}  (-V R_{\lambda,0} )^n
\]
is convergent and equal to the inverse of $H_{V}^\per - \lambda$. This implies that 
\[
\sigma(H_V) \subset ( - (AQ + A^2Q^2), \infty).
\]
Moreover, since
\[
\left\lVert \sum_{n=0}^\infty  (-V R_{\lambda,0} )^n \right\rVert_{2,2} 
\le 
\frac 1{1-1/8} 
= 
\frac 87,
\]
composing this with $R_{\lambda,0} $ and using $\lVert R_{\lambda,0}  \rVert_{2,2} = \lvert \lambda \rvert^{-1}$ and \eqref{freeresolvent} proves the estimates
\begin{align}
\lVert R_{\lambda,V} \rVert_{2,2} & \le \frac 87 \lvert \lambda \rvert^{-1} \label{R1} \\
\lVert R_{\lambda,V} \rVert_{2,\infty} & \le \frac 87 C (\lvert \lambda \rvert^{-1/2} + \lvert \lambda \rvert^{-1}) \label{R2}
\end{align}

The resolvent identity, together with \eqref{R1}, \eqref{R2}, gives
\begin{align*}
\lVert R_{\lambda,V_1} - R_{\lambda,V_2} \rVert_{2,2} 
& \le 
\lVert R_{\lambda,V_1} \rVert_{2,2} \lVert V_1 - V_2 \rVert_{\B} \lVert R_{\lambda,V_2} \rVert_{2,\infty} \\
& \le 
\frac {64}{49} C \lVert V_1 - V_2 \rVert_{\B} (\lvert \lambda\rvert^{-3/2} + \lvert \lambda\rvert^{-2} ).
\end{align*}
By the spectral mapping, $\sigma(R_{\lambda,V_j}) = \{ (E_{n,V_j} - \lambda)^{-1} \mid n \in \mathbb{N}\} \cup \{ 0\}$ so by the variational principle,
\begin{equation}\label{R3}
\left\lvert \frac 1{E_{n,V_1} - \lambda} - \frac 1{E_{n,V_2} - \lambda} \right\rvert 
\le 
\lVert R_{\lambda,V_1} - R_{\lambda,V_2} \rVert_{2,2} 
\le  
\frac {64}{49} C \lVert V_1 - V_2 \rVert_{\B}  (\lvert \lambda\rvert^{-3/2} + \lvert \lambda\rvert^{-2} ).
\end{equation}
To turn this into an estimate for $\lvert E_{n,V_1} - E_{n,V_2} \rvert$, we set
\[
\lambda = - \max\left\{ 2 \lvert E_{n,V_2} \rvert, A Q + A^2 Q^2 \right\}.
\]
Then
\[
\lvert {E_{n,V_2} - \lambda} \rvert  \frac {64}{49} C \lVert V_1 - V_2 \rVert_{\B}  (\lvert \lambda\rvert^{-3/2} + \lvert \lambda\rvert^{-2} ) \le \frac 32 \lvert \lambda \rvert \frac {64}{49}  C (2Q) \lvert \lambda\rvert^{-1} \frac 2{AQ} < \frac 12,
\]
so \eqref{R3} is easily seen to imply
\begin{align*}
\lvert E_{n,V_1} - E_{n,V_2} \rvert & \le 2 \lvert E_{n,V_2} - \lambda \rvert^2  \frac {64}{49} C \lVert V_1 - V_2 \rVert_2  (\lvert \lambda\rvert^{-3/2} + \lvert \lambda\rvert^{-2} ) \\
& \le 2 \frac{9}4 \lvert \lambda \rvert^2 \frac {64}{49} C (\lvert \lambda\rvert^{-3/2} + \lvert \lambda\rvert^{-2} ) \lVert V_1 - V_2 \rVert_{\B}  \\
& \le  9C (\lvert \lambda\rvert^{1/2} +1 ) \lVert V_1 - V_2 \rVert_{\B}
\end{align*}
which completes the proof by our choice of $\lambda$.
\end{proof}

\section{Convergence of Periodic Spectra} \label{sec:specconv}

The results of Section~\ref{sec:bandstab} control the stability of bands of periodic spectra under $L^2$ perturbations. This will enable us to prove effective estimates on the measure of the spectrum in passing from $V_\ell$ to $V_{\ell+1}$. The next key ingredient one needs to control is the relationship between the spectra of the periodic approximants and the spectrum of the limiting potential. This will enable us to push our step-by-step estimates through to the limit. If $V_n \to V$ uniformly, then this is easy, since an unbounded version of Lemma~\ref{l:spec:pert} implies that $\sigma(H_{V_n}) \to \sigma(H_V)$ in the Hausdorff metric. However, in our setting, we only have $V_n \to V$ in a uniform local $L^2$ sense, so the convergence statement on the spectra is weaker. Nevertheless, we still have norm convergence of resolvents, which enables us to establish a ``pseudo local Hausdorff'' convergence result for the spectra which is strong enough to get lower bounds on the limiting spectrum.

First, let us recall the definition of the Hausdorff metric on compact subsets of $\R$.

\begin{defi}
Given two compact sets $F,K \subseteq \R$, the \emph{Hausdorff distance} between them is defined by
\begin{equation}\label{eq:hdimdef}
d_{\mathrm H}(F,K)
=
\inf \{ r > 0 : F \subseteq B_r(K) \text{ and } K \subseteq B_r(F) \}.
\end{equation}
In \eqref{eq:hdimdef}, $B_r(X)$ denotes the open $r$-neighborhood of the set $X \subseteq \R$.
\end{defi}

The following relationship between spectra of bounded, self-adjoint operators and the Hausdorff metric is well-known.

\begin{lemma}\label{l:spec:pert} If $A$ and $B$ are bounded self-adjoint operators, then
$$
d_{\mathrm H}(\sigma(A), \sigma(B))
\leq
\|A - B\|.
$$
\end{lemma}

\begin{lemma} \label{l:ptspec:conv}
Given $V \in \PT(\R)$ with periodic approximants $(V_j)_{\j=1}^\infty$, let $H = -\Delta + V$ and $H_j = -\Delta + V_j$, and put $\Sigma = \sigma(H)$ and $\Sigma_j = \sigma(H_j)$. If $(I_j)_{j=1}^\infty$ is a family of compact subsets of $\R$ which converges to the compact set $I \subseteq \R$ in the Hausdorff metric, then, for every $\varepsilon > 0$, there exists $N \in \Z_+$ such that
$$
I_j \cap \Sigma_j
\subseteq
B_\varepsilon(I \cap \Sigma)
$$
whenever $j \geq N$. In particular,
\begin{equation}\label{eq:lebsemicont}
|I \cap \Sigma|
\geq
\limsup_{j \to \infty} |I_j \cap \Sigma_j|,
\end{equation}
where $|\cdot|$ denotes Lebesgue measure on $\R$.
\end{lemma}

\begin{proof}
This follows from the spectral mapping theorem and norm-resolvent convergence. To begin, notice that the bottoms of the spectra converge; that is,
$$
E_1
:=
\lim_{j \to \infty} E_{1,j}
$$
exists, where $E_{1,j} = E_{1,V_j} = \inf\Sigma_j$ for each $j \geq 1$. This is a consequence of the Pastur--Tkachenko condition \eqref{eq:pt:def} and the stability condition \eqref{eq:edgestability} from Theorem~\ref{Tbandstability} with $n = 1$. In particular,
$$
\lambda_0
:=
\inf_{j \geq 1} E_{1,j}
-1
$$
is finite. Since $\mathrm{dist}(\lambda_0,\sigma(H_j)) \geq 1$ for every $j$, we may apply \cite[Lemma~17.17]{PasturFigotin} to see that $H_j$ converges to $H$ in the norm-resolvent sense, i.e.,
$$
R
=
\lim_{j \to \infty} R_j,
$$
where $R = (H-\lambda_0)^{-1}$ and $R_j = (H_j - \lambda_0)^{-1}$. Consequently, Lemma~\ref{l:spec:pert} implies that
$$
\sigma(R)
=
\lim_{j \to \infty} \sigma(R_j),
$$
where the limit is taken in with respect to the Hausdorff metric on compact subsets of $\R$. By the spectral mapping theorem for unbounded operators, we have
$$
\Sigma_j'
:=
\sigma(R_j)
=
\overline{g(\Sigma_j)}
\text{ for every } j \geq 1, \text{ and }
\Sigma'
:=
\sigma(R)
=
\overline{g(\Sigma)},
$$
with $g(z) = (z - \lambda_0)^{-1}$; see \cite[Theorem~3.18]{Teschl2000}. Notice that $g$ is invertible with inverse $h(z) = z^{-1} +\lambda_0$. Without loss of generality, we may assume $I_j \subseteq [\lambda_0 + 1,\infty)$ for each $j$. It is easy to see that $I_j' := g(I_j)$ is a compact subset of $(0,1]$ for each $j$. Given $\delta > 0$ Hausdorff convergence of $\Sigma_j'$ implies
$$
I_j' \cap \Sigma_j'
\subseteq
B_\delta(I' \cap \Sigma')
$$
for all sufficiently large $j$. Of course, $h$ is  Lipschitz on $I'$, so this implies (for some $I$-dependent constant $\kappa > 0$):
$$
I_j \cap \Sigma_j
\subseteq
B_{\kappa\delta}(I \cap \Sigma)
$$
for all sufficiently large $j$. The semicontinuity statement \eqref{eq:lebsemicont} follows by sending $\delta \to 0$.

\end{proof}

\section{Proof of Theorem \ref{t:pt}} \label{sec:proofs}

We can now combine the results of the previous three sections to prove the main theorem. The general strategy is as follows:
\begin{itemize}
\item Use Corollary~\ref{coro:breakct} to obtain a lower bound on the lengths of bands of the spectra of the periodic approximants.
\item Use Theorem~\ref{Tbandstability} to obtain an upper bound on the lengths of gaps of the spectra of the periodic approximants.
\item These two items give us effective upper bounds on quantities of the form $|I \cap \Sigma_\ell \setminus \Sigma_{\ell+1}|$ with $I$ a compact interval.
\item Use Lemma~\ref{l:ptspec:conv} to push the estimates through to the limit.
\end{itemize}

\begin{proof}[Proof of Theorem~\ref{t:pt}]
Suppose $V \in \PT(\R)$, and let $0 < \tau < 1$ be given. If $V$ is periodic, the conclusion of the theorem is trivial, so assume that $V$ is aperiodic. Let $(V_n)_{n=1}^\infty$ be a sequence of periodic potentials, such that $V_n$ is $T_n$-periodic, $T_n$ divides $T_{n+1}$ for each $n$, and $\|V - V_n\|_{\St}$ is $o \! \left(e^{-bT_{n+1}}\right)$ for every $b > 0$. Notice that this property (i.e.\ equation \eqref{eq:pt:def}) is preserved if one removes finitely many terms of the sequence $(V_n)_{n=1}^\infty$ and consecutively renumbers the resulting sequence. In particular, it is no loss of generality to assume $T_n \geq 1$ for every $n$. For each $n \in \Z_+$, denote $\Sigma_n = \sigma(H_{V_n})$, and put $\Sigma = \sigma(H_V )$. Repeating the argument from the beginning of the proof of Lemma~\ref{l:ptspec:conv}, we see that $\sup_n \lvert \inf \Sigma_n \rvert$ is finite. Since spectral homogeneity is preserved by adding a constant to the potential, this means that we may assume without loss of generality that $\Sigma_n \subseteq[1,\infty)$ for every $n$. Now, define
\begin{align*}
Q
& =
\sup_{n \geq 1} \|V_n\|_{\B} \\
K
& =
\max\left( C, C_1,  Q, CQ^{1/2}, 8 \right),
\end{align*}
where $C$ is the universal constant from Lemma~\ref{l:discderivest} and $C_1$ is the universal constant from Theorem~\ref{Tbandstability}. Notice that $K$ depends only on $Q$. Using the Pastur--Tkachenko condition \eqref{eq:pt:def}, it is easy to see that
$$
\sum_{n=1}^\infty T_{n+1}^6 e^{K T_{n+1}} \|V_n - V_{n+1}\|_{\B}
<
+\infty,
$$
Thus, by removing finitely many terms of the sequence $(V_n)_{n=1}^\infty$ and consecutively renumbering once more, we may assume that
\begin{equation} \label{eq:smalltail}
\sum_{n=1}^\infty
T_{n+1}^6 e^{KT_{n+1}} \|V_n - V_{n+1}\|_{\B} 
<
\frac{1 - \tau}{3K^4}.
\end{equation}
Put
$$
\delta_0 
= 
\min\left(K^{-1} T_1^{-3} e^{-KT_1}, \frac{1 - \tau}3 \right).
$$
Following the general strategy of \cite{F14}, we will prove the following estimate:
\begin{equation} \label{eq:sbs:homog:est}
|B_\delta(x) \cap \Sigma_N|
\geq
\delta \tau
\text{ for all } 
x \in \Sigma_N 
\text{ and every } 
0 < \delta \leq \delta_0
\end{equation}
for all $N \in \Z_+$. To that end, fix $N \in \Z_+$, $x \in \Sigma_N$, and $0 < \delta \leq \delta_0$. For notational simplicity, put $s = s(x) := 1+\sqrt{x}$. If $\delta \leq sK^{-1} T_N^{-3} e^{-K T_N}$, \eqref{eq:sbs:homog:est} is an obvious consequence of Corollary~\ref{coro:breakct}. Specifically, $\delta$ is less than the length of the band of $\Sigma_N$ which contains $x$, so
$$
|B_{\delta}(x) \cap \Sigma_N|
\geq
\delta
$$ 
in this case. Notice that we have used the assumption $\inf\Sigma_N \geq 0$ to get rid of $E_0$ from the statement of Corollary~\ref{coro:breakct}. Otherwise, $\delta > sK^{-1} T_N^{-3} e^{-K T_N}$, and there is a unique integer $n$ with $1 \leq n \leq N-1$ such that
\begin{equation}\label{deltachoice}
sK^{-1} T_{n+1}^{-3} e^{-K T_{n+1}} 
< 
\delta 
\leq 
sK^{-1} T_{n}^{-3} e^{-K T_n}.
\end{equation}
This integer $n$ is relevant, as it determines the periodic approximant corresponding to the length scale $\delta$ near and above $x$. More precisely, by Corollary~\ref{coro:breakct}, any band of $\Sigma_n$ which is contained in $[x - 1, \infty)$ has length at least $\delta$. By Theorem~\ref{Tbandstability}, there exists $x_0 \in \Sigma_n$ with
\begin{equation}
|x-x_0| 
\leq 
s K^3 \sum_{\ell = n}^{N - 1}
T_{\ell+1}^3 \|V_\ell - V_{\ell+1}\|_{\B}
<
1.
\end{equation}
This bound follows by a straightforward backward induction on $n$, starting from $n=N$. The inductive step uses Theorem~\ref{Tbandstability} and the estimate $1+|x_0|^{1/2} \le 2 s$, which follows from the inductive assumption. Using \eqref{deltachoice}, we deduce
\begin{align}
\nonumber
\lvert x-x_0 \rvert
& \leq 
s K^3  \sum_{\ell = n}^{N - 1}
T_{\ell + 1}^3 \|V_\ell - V_{\ell+1}\|_{\B} \\
\nonumber
& < \delta e^{K T_{n+1}} T_{n + 1}^3 K^4  \sum_{\ell= n}^{N - 1} T_{\ell + 1}^3 \|V_\ell - V_{\ell+1}\|_{\B} \\
\nonumber
& < \delta K^4  \sum_{\ell= n}^{N - 1} e^{K T_{\ell+1}} T_{\ell + 1}^6 \|V_\ell - V_{\ell+1}\|_{\B} \\
\label{eq:expsum2}
& < \delta \frac{1-\tau}3.
\end{align}
Thus, there exists an interval $I_0$ with $x_0 \in I_0 \subseteq B_{\delta}(x) \cap \Sigma_n$ such that 
$$
|I_0| 
=
\delta - \frac{1-\tau}{3}\delta
=
\frac{2+\tau}{3}\delta.
$$
By standard measure theory, we have
$$
|B_\delta(x) \cap \Sigma_N|
\geq
|I_0 \cap \Sigma_n| - \sum_{\ell = n}^{N - 1}|I_0 \cap (\Sigma_\ell \setminus \Sigma_{\ell + 1})|
$$
By Corollary~\ref{coro:breakct}, the interval $I_0$ completely contains at most $\delta K s^{-1} T_{\ell+1}^{3} e^{K T_{\ell+1}}$ bands of $\Sigma_{\ell + 1}$ for each $\ell \geq n$. Consequently, 
\begin{align*}
|I_0 \cap (\Sigma_\ell \setminus \Sigma_{\ell+1})|
& \leq
K^3 T_{\ell+1}^3 s \left(\delta K s^{-1} T_{\ell+1}^{3} e^{KT_{\ell+1}} + 1 \right) \cdot
\| V_\ell - V_{\ell + 1} \|_{\B}  \\
& \leq
2 \delta K^4 T_{\ell+1}^6 e^{KT_{\ell+1}}  \cdot
\| V_\ell - V_{\ell + 1} \|_{\B} 
\end{align*}
for each $\ell \geq n$, by Theorem~\ref{Tbandstability} and \eqref{deltachoice}. Summing this over $\ell$ and estimating the result with \eqref{eq:smalltail}, we obtain
\begin{align*}
\sum_{\ell = n}^{N - 1} |I_0 \cap (\Sigma_\ell \setminus \Sigma_{\ell + 1} )|
& \leq
\sum_{\ell=n}^{N-1} 2 \delta K^4 T_{\ell+1}^6 e^{KT_{\ell+1}}  \cdot
\| V_\ell - V_{\ell + 1} \|_{\B}  \\
& < 2 \delta \frac{1-\tau}3.
\end{align*}
Putting all of this together, we have
\begin{align*}
|B_\delta(x) \cap \Sigma_N|
& \geq
|I_0 \cap \Sigma_n| - \sum_{\ell = n}^{N-1}|I_0 \cap(\Sigma_\ell \setminus \Sigma_{\ell + 1})| \\
& >
\frac{2+\tau}{3} \delta 
- \frac{2 - 2\tau}{3} \delta \\
& =
\tau \delta.
\end{align*}
This proves \eqref{eq:sbs:homog:est} for arbitrary $N \in \Z_+$. Using Lemma~\ref{l:ptspec:conv}, we obtain
$$
|B_\delta(x) \cap \Sigma|
\geq
\tau\delta
\text{ for all } x \in \Sigma, \text{ and } 0 < \delta \leq \delta_0,
$$
to wit, $\Sigma$ is homogeneous.
\end{proof}

\section{The Discrete Setting} \label{sec:discrete}

In this section, we will describe how to extend Theorem~\ref{t:pt} to Jacobi and CMV matrices. In this case, the discrete setting is markedly easier than the continuum setting, so we will only sketch the proof. Recall that a Jacobi matrix is an operator of the form $J = J_{a,b}$ on $\ell^2(\Z)$, defined by
$$
(J\varphi)_n
=
a_{n-1} \varphi_{n-1} + a_n \varphi_{n+1} + b_n\varphi_n,
\quad
\varphi \in \ell^2(\Z), \, n \in \Z,
$$
where $a$ and $b$ are bounded, real-valued sequences. A CMV matrix is a unitary operator on $\ell^2(\Z)$ which has a matrix representation of the form
$$
\small
\E
=
\begin{pmatrix}
\ddots & \ddots & \ddots &&&&&  \\
\overline{\alpha_0}\rho_{-1} & -\overline{\alpha_0}\alpha_{-1} & \overline{\alpha_1}\rho_0 & \rho_1\rho_0 &&& & \\
\rho_0\rho_{-1} & -\rho_0\alpha_{-1} & -\overline{\alpha_1}\alpha_0 & -\rho_1 \alpha_0 &&& & \\
&  & \overline{\alpha_2}\rho_1 & -\overline{\alpha_2}\alpha_1 & \overline{\alpha_3} \rho_2 & \rho_3\rho_2 & & \\
& & \rho_2\rho_1 & -\rho_2\alpha_1 & -\overline{\alpha_3}\alpha_2 & -\rho_3\alpha_2 &  &  \\
& &&& \overline{\alpha_4} \rho_3 & -\overline{\alpha_4}\alpha_3 & \overline{\alpha_5}\rho_4 & \rho_5\rho_4 \\
& &&& \rho_4\rho_3 & -\rho_4\alpha_3 & -\overline{\alpha_5}\alpha_4 & -\rho_5 \alpha_4  \\
& &&&& \ddots & \ddots &  \ddots
\end{pmatrix},
$$
where $\alpha_n \in \D$ for all $n \in \Z$, and $\rho_n = (1 - |\alpha_n|^2)^{1/2}$. 

 Let us say that a Jacobi matrix $J=J_{a,b}$ is of Pastur--Tkachenko type if there exists $\gamma > 0$ such that $a_n \geq \gamma$ for all $n$ and there are $p_n$-periodic Jacobi matrices $J^{(n)}$ such that
$$
\lim_{n \to \infty} e^{Cp_{n+1}} \| J^{(n)} - J \|_\infty
=
0
$$
for all $C > 0$. Such matrices were studied in \cite{egorova}. Similarly, a CMV matrix $\E$ is of Pastur--Tkachenko type if there are periodic CMV matrices $\E^{(n)}$ such that
$$
\lim_{n \to \infty} e^{Cp_{n+1}} \| \mathcal E^{(n)} -\mathcal E \|_\infty
=
0
$$
for all $C>0$.

\begin{theorem} \label{t:pt:disc}
If $J$ is a Jacobi matrix of Pastur--Tkachenko type, then $\sigma(J)$ is $\tau$-homogeneous for every $\tau \in (0,1)$. Similarly, if $\E$ is a CMV matrix of PT type, then $\sigma(\E)$ is a $\tau$-homogeneous subset of $\partial \D$ for all $\tau \in (0,1)$.
\end{theorem}

\begin{proof}
Simply follow the proof of Theorem~\ref{t:pt}. In particular, there are three main ingredients which are used in the proof:
\begin{enumerate}
\item The discriminant estimate from Lemma~\ref{l:discderivest}
\item The band edge stability estimate from Theorem~\ref{Tbandstability}.
\item The spectral convergence result from Lemma~\ref{l:ptspec:conv}.
\end{enumerate}

The analog of (1) in this setting is elementary. In this case, the monodromy matrix is a product of $p$ one-step transfer matrices. Taking the derivative using the product rule, and estimating the resulting terms gives an analogous estimate for the derivative of the Jacobi/CMV discriminants on their spectra.

Since the periodic approximants converge \emph{uniformly} to the limiting operator, the stability estimates from (2) can be replaced by the $\ell^\infty$ theory, which is much simpler. More precisely, in the Jacobi case, $\| J^{(n)} - J^{(n+1)}\| = \delta$ implies that the each band edge of $J^{(n+1)}$ is within $\delta$ of a band edge of $J^{(n)}$, where the latter is thought of as a degenerate $p_{n+1}$-periodic operator. In the language of \cite{F14}, every band edge of $J^{(n+1)}$ is within $\delta$ of a $(p_{n+1}/p_n)$-break point of $J^{(n)}$.

Finally, since the operators converge uniformly in this case, we can replace (3) by honest Hausdorff convergence of the approximating spectra, which allows us to push the step-by-step estimates through to the limiting spectrum via a semicontinuity argument as before.

\end{proof}

\end{document}